\documentclass{amsart}

\usepackage{amsmath,amssymb,latexsym,soul, tikz}
\usepackage{cite,mathrsfs,amsfonts, verbatim}
\usepackage{color,enumitem,graphicx}
\usepackage{esint}
\usepackage[colorlinks=true,urlcolor=blue,
citecolor=red,linkcolor=blue,linktocpage,pdfpagelabels,
bookmarksnumbered,bookmarksopen]{hyperref}
\usepackage[english]{babel}

\usepackage[left=3cm,right=3cm,top=3cm,bottom=3cm]{geometry}

\usepackage[hyperpageref]{backref}
\usepackage{mathrsfs}
\usepackage{lmodern}
\numberwithin{equation}{section}

\newtheorem{theorem}{Theorem}

\newtheorem{proposition}[theorem]{Proposition}
\newtheorem{corollary}[theorem]{Corollary}

\newtheorem{example}[theorem]{Example}

\DeclareMathOperator*{\divergenz}{div}

\DeclareMathOperator*{\essinf}{ess ~inf}

\newcommand{\N}{\mathbb{N}}

\newcommand{\R}{\mathbb{R}}

\newcommand{\RN}{\mathbb{R}^N}

\newcommand{\Wpzero}[1]{W^{1,#1}_0(\Omega)}

\newcommand{\rand}{\partial\Omega}

\newcommand{\into}{\int_{\Omega}}

\newcommand{\s}{\cdot}

\newcommand{\close}{\overline{\Omega}}

\numberwithin{theorem}{section}
\numberwithin{equation}{section}

\begin{document}
\title[Concave-convex and parametric $(p,q)$-equations]{On a Dirichlet problem with $(p,q)$-Laplacian and parametric concave-convex nonlinearity}
\begin{abstract}
A homogeneous Dirichlet problem with $(p,q)$-Laplace differential operator and reaction given by a parametric $p$-convex term plus a $q$-concave one is investigated.  A bifurcation-type result, describing changes in the set of positive solutions as the parameter $\lambda>0$ varies, is proven. Since for every admissible $\lambda$ the problem has a smallest positive solution $\bar u_{\lambda}$, both monotonicity and continuity of the map $ \lambda \mapsto \bar u_{\lambda}$ are studied. 
\end{abstract}
\author{Salvatore A. Marano}
\address{Salvatore A. Marano\\
Dipartimento di Matematica e Informatica\\
Universit\`{a} degli Studi di Catania \
Viale A. Doria 6, 95125 Catania, Italy}
\email{marano@dmi.unict.it}
\author{Greta Marino}
\address{Greta Marino\\
Dipartimento di Matematica e Informatica \\
Universit\`a degli Studi di Catania \\
Viale A. Doria 6, 95125 Catania, Italy}
\email{greta.marino@dmi.unict.it}
\author{Nikolaos S. Papageorgiou}
\address{Nikolaos S. Papageorgiou\\
Department of Mathematics\\
National Technical University\\
Zografou Campus, 15780 Athens, Greece}
\email{npapg@math.ntua.gr}
\maketitle
\section{Introduction}
Let $\Omega$ be a bounded domain in $\R^N$ with a $C^2$-boundary $\rand$, let $1<\tau< q< p<+\infty$, and let $f:\Omega\times\R\to\R$ be a Carath\'{e}odory function. Consider the Dirichlet problem
\begin{equation}\label{problem}
\tag{${\rm P}_{\lambda}$}
\left\{
\begin{array}{ll}
-\Delta_p u- \Delta_q u= u^{\tau-1}+ \lambda f(x, u) & \mbox{in $\Omega$,}\cr
u>0 & \mbox{in $\Omega$,}\cr
u=0 & \mbox{on $\rand$,}\cr
\end{array}
\right.
\end{equation}
where $\lambda>0$ is a parameter while $\Delta_r$, $r>1$, denotes the $r$-Laplacian, namely 
$$\Delta_r u:= \divergenz(\vert\nabla u\vert^{r-2}\nabla u)\quad\forall\, u\in\Wpzero{r}.$$
The nonhomogeneous differential operator $Au:=\Delta_p u+\Delta _q u$ that drives \eqref{problem} is usually called $(p,q)$-Laplacian. It stems from a wide range of important applications, including models of elementary particles \cite{De}, biophysics \cite{Fi}, plasma physics \cite{W}, reaction-diffusion equations \cite{CI}, elasticity theory \cite{Z}, etc. That's why the relevant literature looks daily increasing and numerous meaningful works on this subject are by now available; see the survey paper \cite{MM} for a larger bibliography.

Since $\tau<q< p$, the function $\xi\mapsto \xi^{\tau-1}$ grows $(q-1)$-sublinearly at $+\infty$, whereas $\xi\mapsto f(x,\xi)$ is assumed to be $(p-1)$-superlinear near $+\infty$, although it need not satisfy the usual (in such cases) Ambrosetti-Rabinowitz condition. So, the reaction in \eqref{problem} exhibits the competing effects of concave and convex terms, with the latter multiplied by a positive parameter.

The aim of this paper is to investigate how the solution set of \eqref{problem} changes as $\lambda$ varies. In particular, we prove that there exists a critical parameter value $\lambda^*>0$ for which problem \eqref{problem} admits
\begin{itemize}
\item at least  two solutions if $\lambda\in(0,\lambda^*)$,
\item at least one solution when $\lambda=\lambda^*$, and
\item no solution provided $\lambda>\lambda^* $.
\end{itemize}
Moreover, we detect a smallest positive solution $\bar u_{\lambda}$ for each $\lambda\in(0,\lambda^*]$ and show that the map $\lambda\mapsto\bar u_{\lambda}$ turns out left-continuous, besides increasing. 

The first bifurcation result for semilinear Dirichlet problems driven by the Laplace operator was established, more than twenty years ago, in the seminal paper\cite{ABC} and then extended to the $p$-Laplacian in \cite{GMP,GZ}. These works treat the reaction
$$\xi\mapsto\lambda\xi^{s-1}+\xi^{r-1},\quad\xi\geq 0,$$
where $1< s< p< r< p^* $, $\lambda>0$, and $p^*$ denotes the critical Sobolev exponent. A wider class of nonlinearities has recently been investigated in \cite{MP}, while \cite{PRR} deals with Robin boundary conditions. It should be noted that, unlike our case, $\lambda$ always multiplies the concave term, which changes the analysis of the problem.  Finally, \cite{BLP,GP,PR} contain analogous bifurcation theorems for problems of a different kind, whereas \cite{MMP,MMP1} study $(p,q)$-Laplace equations having merely concave right-hand side.

Our approach is based on the critical point theory, combined with appropriate truncation and comparison techniques. 
\section{Mathematical background and hypotheses}\label{prel}
Let $(X,\Vert\cdot\Vert)$ be a real Banach space. Given a set $V\subseteq X$, write $\overline{V}$ for the closure of $V$, $\partial V$ for the boundary of $V$, and ${\rm int}_X(V)$ or simply ${\rm int}(V)$, when no confusion can arise, for the interior of $V$. If $x\in X$ and $\delta>0$ then
$$B_\delta(x):=\{ z\in X:\;\Vert z-x\Vert<\delta\}\, ,\quad B_\delta:=B_\delta(0)\, .$$
The symbol $(X^*,\Vert\cdot \Vert_{X^*})$ denotes the dual space of $X$, $\langle\s \,,\s\rangle$ indicates the duality pairing between $X$ and $X^*$, while $x_n\to x$ (respectively, $x_n\rightharpoonup x$) in $X$ means `the sequence $\{x_n\}$ converges strongly (respectively, weakly) in $X$'. We say that $A:X\to X^*$ is of type $({\rm S})_+$ provided
$$x_n\rightharpoonup x\;\mbox{ in }\; X,\quad
\limsup_{n\to+\infty}\langle A(x_n),x_n-x\rangle\leq 0\quad\implies\quad x_n\to x.$$
The function $\Phi:X\to\mathbb{R}$ is called coercive if $\displaystyle{\lim_{\Vert x\Vert\to+\infty}}\Phi(x)=+\infty$ and weakly sequentially lower semicontinuous when
$$x_n\rightharpoonup x\;\mbox{ in }\; X\quad\implies\quad\Phi(x)\leq\liminf_{n\to\infty}\Phi(x_n).$$
Suppose $\Phi\in C^1(X)$. We denote  by $K(\Phi)$ the critical set of $\Phi$, i.e., 
$$K(\Phi):=\{x\in X:\,\Phi'(x)=0\}.$$
The classical Cerami compactness condition for $\Phi$ reads as follows:
\begin{itemize}
\item[$({\rm C})$] {\it Every $\{x_n\}\subseteq X$ such that $\{\Phi(x_n)\}$ is bounded and  $(1+\Vert x_n\Vert)\Phi'(x_n)\to 0$ in $X^*$ has a convergent subsequence.}
\end{itemize}
From now on, $\Omega$ indicates a fixed bounded domain in $\RN$ with a $C^2$-boundary $\partial\Omega$. Let $u,v:\Omega\to
\R$ be measurable and let $t\in\R$. The symbol $u\leq v$ means $u(x)\leq v(x)$ for almost every $x\in\Omega$, $t^\pm:=\max\{\pm t,0\}$, $u^\pm(\cdot):= u(\cdot)^\pm$. If $u,v$ belong to a function space, say $Y$, then we set
$$[u,v]:=\{ w\in Y:u\leq w\leq v\}\, ,\quad [u):=\{w\in Y:u\leq w\}\, .$$
The conjugate exponent $r'$ of a number $r\geq 1$ is defined by $r':=r/(r-1)$, while $r^*$ indicates its Sobolev conjugate, namely
\begin{equation*}
r^*:=
\begin{cases}
\frac{Nr}{N-r}&\text{when $r<N$},\\
+\infty&\text{otherwise}.
\end{cases}
\end{equation*}
As usual,
$$\Vert u\Vert_r:= \left(\int_{\Omega} |u|^r\, dx\right)^{1/r}\;\forall\, u\in L^{r}(\Omega),\quad
\|u\|_{1,r}:=\left(\int_{\Omega} |\nabla u|^r\, dx\right)^{1/r}\;\forall\, u\in W^{1,r}_0(\Omega),$$
and $W^{-1,r'}(\Omega)$ denotes the dual space of $W^{1,r}_0(\Omega)$. We will also employ the linear space $C^1_0(\overline{\Omega}):=\{u\in C^1(\overline{\Omega}): u\lfloor_{\partial\Omega}=0\}$, which is complete with respect to the standard $C^1(\overline{\Omega})$-norm. Its positive cone
$$C_+:=\{u\in C^1_0(\overline{\Omega}):u(x)\geq 0\text{ in $\overline{\Omega}$}\}$$
has a nonempty interior given by
\[
{\rm int } (C_+)=\left\{u\in C_+: u(x)>0\;\;\forall\, x\in\Omega,\;\frac{\partial u}{\partial n}(x)<0\;\;\forall\, x\in\partial\Omega
\right\}.
\]
Here $n(x)$ denotes the outward unit normal to $\partial\Omega$ at $x$.

Let $A_r:W^{1,r}_0(\Omega)\to W^{-1,r'}(\Omega)$ be the nonlinear operator stemming from the negative $r$-Laplacian, i.e.,
$$\langle A_r(u),v\rangle:=\int_\Omega|\nabla u|^{r-2}\nabla u\s\nabla v\, dx\, ,\quad u,v\in W^{1,r}_0(\Omega)\, .$$
We know \cite[Section 6.2]{GP1} that $A_r$ is bounded, continuous, strictly monotone, and of type $({\rm S})_+$. The Liusternik-Schnirelmann theory gives an increasing sequence $\{\lambda_{n, r}\}$ of eigenvalues for $A_r$. The following assertions can be found in \cite[Section 6.2]{GP1}.
\begin{itemize}
\item[$({\rm p}_1)$] $\lambda_{1,r}$ is positive, isolated, and simple.
\item[$({\rm p}_2)$] $\Vert u\Vert_r^r\leq\displaystyle{\frac{1}{\lambda_{1,r}}}
\Vert u\Vert_{1,r}^r$ for all $u\in W^{1,r}_0(\Omega)$.
\item[$({\rm p}_3)$]  $\lambda_{1,r}$ admits an eigenfunction $\phi_{1,r}\in{\rm int}(C_+)$ such that $\Vert\phi_{1,r}\Vert_r=1$.
\end{itemize}
Proposition 13 of \cite{BoTa} then ensures that
\begin{itemize}
\item[$({\rm p}_4)$] If $r\neq \hat r$ then  $\phi_{1,r}$ and $\phi_{1,\hat r}$ are linearly independent.
\end{itemize}
Let  $g:\Omega\times\R\to\R$ be a Carath\'eodory function satisfying the growth condition
$$\vert g(x, t)\vert\le a(x) \left(1+\vert t\vert^{s-1}\right)\;\;\text{in}\;\;\Omega\times\R,$$
where $a\in L^{\infty}(\R)$, $ 1< s\le p^*$. Set $G(x, \xi):= \int_0^\xi g(x, t)\, dt$ and consider the $C^1$-functional
$\varphi:\Wpzero{p}\to\R$ defined by
\[
\varphi(u):=\frac{1}{p}\Vert\nabla u\Vert_p^p+ \frac{1}{q}\Vert\nabla u\Vert_q^q- \into G(x, u(x))\, dx,\quad  u\in \Wpzero{p}.
\]
\begin{proposition}[\cite{GP-SVAN}, Proposition 2.6] 
\label{prop2}
If $u_0 \in \Wpzero{p} $ is a local $ C^1_0(\close)$-minimizer of $\varphi$ then $u_0\in C^{1,\alpha}(\close)$ for some $\alpha\in (0,1)$ and $u_0 $ turns out to be a local $ \Wpzero{p}$-minimizer of $ \varphi$.
\end{proposition}
Combining this result with the strong comparison principle below, essentially due to Arcoya-Ruiz \cite{AR}, shows that certain constrained minimizers actually are `global' critical points. Recall that, given $ h_1,h_2\in L^{\infty}(\Omega),$ 
$$h_1\prec h_2\;\iff\;\essinf_{K}\,(h_2- h_1)>0\;\; \text{for any nonempty compact set}\; K\subseteq\Omega.$$
\begin{proposition}\label{arcoya-ruiz}
Let $a\in\R_+$, $h_1,h_2\in L^{\infty}(\Omega)$, $u_1\in C^1_0(\close)$, $u_2\in{\rm int}(C_+)$. Suppose $h_1\prec h_2$ as well as
\[
-\Delta_p u_i-\Delta_q u_i+a\vert u_i\vert^{p-2} u_i= h_i\;\;\text{in}\;\;\Omega,\;\; i=1,2.
\]
Then, $u_2-u_1\in{\rm int}(C_+)$. 
\end{proposition}
Throughout the paper, `for every $x\in\Omega$' will take the place of `for almost every $x\in\Omega$', $c_0,c_1,\ldots$ indicate suitable positive constants, $f:\Omega\times\R\to\R$ is a Carath\'eodory function such that $ f(\cdot \,, t)=0$ provided $t\leq 0$, while $F(x,\xi):=\int_0^\xi f(x,t)\, dt$.

The following hypotheses will be posited.

\begin{itemize}
\item[$({\rm h}_1)$] There exist $\theta\in[\tau, q]$ and $r \in (p, p^*) $ such that
$$c_1 t^{p-1}+c_2 t^{q-1}\le f(x, t) \le c_0 \left( t^{\theta-1}+ t^{r-1}\right) \;\;\forall\, (x,t)\in\Omega\times\R_+\, ,$$
where $c_2>\lambda_{1,q}$.
\item[$({\rm h}_2)$] $\lim\limits_{\xi\to+\infty}\frac{F(x, \xi)}{\xi^p}=+\infty$ uniformly with respect to $x\in\Omega$.
\item[$({\rm h}_3)$] $\liminf\limits_{\xi\to+\infty}\frac{f(x, \xi)\xi- pF(x, \xi)}{\xi^{\beta}}\geq c_3$ uniformly in $x\in\Omega$. Here, $\beta>\tau$ and
$$(r-p) \max\left\{Np^{-1},1\right\}<\beta< p^*.$$
\item[$({\rm h}_4)$] To every $\rho>0$ there corresponds $\mu_{\rho}> 0$ such that $t\mapsto f(x,t)+\mu_{\rho} t^{p-1} $ is nondecreasing in $[0, \rho]$ for any $x\in\Omega$. 
\end{itemize}
By $({\rm h}_2)$--$({\rm h}_3)$ the perturbation $f(x,\s)$ is $(p-1)$-superlinear at $+\infty$. In the literature, one usually treats this case via the well-known Ambrosetti-Rabinowitz condition, namely:
\begin{itemize}
\item[({\rm AR})] With appropriate $ M> 0$, $ \sigma> p$ one has both $\essinf\limits_{\Omega} F(\s \,, M)>0$ and 
\begin{equation}\label{AR}
0< \sigma F(x,\xi)\le f(x,\xi)\xi\, ,\quad (x,\xi)\in\Omega\times [M,+\infty).
\end{equation}
\end{itemize}
It easily entails $c_3\xi^\sigma\le F(x,\xi)$ in $\Omega\times [M,+\infty)$, which forces $({\rm h}_2)$. However, nonlinearities having a growth rate `slower' than $t^{\sigma-1}$ at $+\infty$ are excluded from \eqref{AR}. Thus, assumption $({\rm h}_3)$ incorporates in our framework more situations. 
\begin{example}
Let $c_2>\lambda_{1,q}$. The functions $f_1,f_2:\R_+\to\R$ defined by
\[
f_1(t):=
\begin{cases}
t^{p-1}+c_2 t^{\tau-1} \quad & \text{if }\; 0 \le t \le 1,\\
t^{r-1}+c_2 t^{q-1} & \text{otherwise}, 
\end{cases}
\quad
f_2(t):= t^{p-1}\log(1+t)+c_2 t^{q-1},\quad t\in\R_+, 
\]
satisfy $({\rm h}_1)$--$({\rm h}_4)$. Nevertheless, $f_1$ alone complies with condition (AR). 
\end{example}
\section{A bifurcation-type theorem}
Write $S_\lambda$ for the set of positive solutions to \eqref{problem}. Lieberman's nonlinear regularity theory \cite[p. 320]{L} and Pucci-Serrin's maximum principle \cite[pp. 111,120]{PS} yield
$$S_{\lambda}\subseteq{\rm int}(C_+).$$
Put $\mathcal{L}:=\{\lambda>0:S_\lambda\neq\emptyset\}$. Our first goal is to establish some basic properties of $\mathcal L$. From now on, $X:=\Wpzero{p}$ and $\Vert\cdot\Vert:=\Vert\cdot\Vert_{1,p}$.
\begin{proposition}\label{prop5}
Under $({\rm h}_1)$ one has $\mathcal L\ne\emptyset$.
\end{proposition}
\begin{proof}
Given $\lambda> 0$, consider the $C^1$-functional $\Psi_{\lambda}:\Wpzero{p} \to\R$ defined by
$$\Psi_{\lambda}(u):=\frac{1}{p} \Vert \nabla u \Vert_p^p+\frac{1}{q} \Vert \nabla u \Vert_q^q-\int_\Omega dx\int_0^{u(x)}g_\lambda(t)\, dt\quad\forall\, u\in\Wpzero{p},$$
where
$$ g_{\lambda}(t):= (t^+)^{\tau-1}+ \lambda c_0 \left[(t^+)^{\theta-1}+ (t^+)^{r-1}\right],\quad t\in\R.$$
Evidently, $g_\lambda$ fulfills \eqref{AR} once $\sigma\in (p,r)$ and $M>0$ is big enough. So, condition (C) holds true for
$\Psi_{\lambda}$. Moreover, 
$$u\in{\rm int}(C_+)\;\;\implies\;\;\lim_{t\to+\infty}\Psi_\lambda(tu)=-\infty$$
because $r>p$. Observe next that if $s\in[1,p^*]$ then
$$\Vert u\Vert_s\leq c\Vert u\Vert_{p^*}\leq C\Vert u\Vert\quad\forall\, u\in X,$$
with $C:=C(s,\Omega)$. This easily leads to
\begin{equation}\label{9}
\begin{split}
\Psi_{\lambda}(u) & \ge\frac{1}{p}\Vert u\Vert^p- c_4\Vert u\Vert^{\tau}-\lambda c_5\left[\Vert u\Vert^{\theta}+ \Vert u\Vert^r \right] \\
&= \left[\frac{1}{p}- c_4\Vert u\Vert^{\tau-p}-\lambda c_5\left(\Vert u\Vert^{\theta-p}+\Vert u\Vert^{r-p}\right) \right]
\Vert u \Vert^p,\quad u\in X.
\end{split}
\end{equation}
Let us set, for any $t> 0$,
\begin{equation*}
\gamma_{\lambda}(t):= c_4 t^{\tau-p}+\lambda c_5(t^{\theta-p}+ t^{r-p}),\quad
\hat\gamma_{\lambda}(t):= (c_4+ \lambda c_5) t^{\tau-p}+2\lambda c_5 t^{r-p}.
\end{equation*}
From $\tau\le\theta< p< r$ it follows $\lambda c_5 t^{\theta-p}\le\lambda c_5 \left(t^{\tau-p}+ t^{r-p}\right)$, which implies
\begin{equation}\label{10}
0<\gamma_{\lambda}(t)\le\hat\gamma_{\lambda}(t) \quad \text{in}\quad (0,+\infty).
\end{equation}
Since $\lim\limits_{t\to 0^+}\hat\gamma_{\lambda}(t)=\lim\limits_{t\to+\infty}\hat\gamma_{\lambda}(t)=+\infty$, there exists
$t_0> 0$ satisfying $\hat\gamma_{\lambda}'(t_0)= 0$. One has
$$t_0:=t_0(\lambda):=\left[\frac{(c_4+\lambda c_5)(p-\tau)}{2\lambda c_5(r-p)}\right]^{\frac{1}{r-\tau}}$$
and, via simple calculations, $\lim\limits_{\lambda\to 0^+}\hat\gamma_{\lambda}(t_0)=0$. On account of \eqref{9}--\eqref{10} we can thus find $\lambda_0>0$ such that
$$\Psi_{\lambda}(u) \ge m_{\lambda}>0= \Psi_{\lambda}(0) \quad \text{for all } u\in\partial B(0,t_0),\;\lambda\in (0,\lambda_0).$$
Pick $\lambda\in (0,\lambda_0)$. The mountain pass theorem entails $\Psi_{\lambda}'(\bar u_{\lambda})= 0$ and
$\Psi_{\lambda}(\bar u_{\lambda})\geq m_\lambda$ with appropriate $\bar u_{\lambda}\in X$. Hence,
\begin{equation}\label{15}
\langle A_p(\bar u_{\lambda})+A_q(\bar u_{\lambda}), v\rangle= \into\left[(\bar u_{\lambda}^+)^{\tau-1}
+\lambda c_0 \left((\bar u_{\lambda}^+)^{\theta-1}+ (\bar u_{\lambda}^+)^{r-1}\right)\right] v\, dx, \quad v\in X,
\end{equation}
and $\bar u_{\lambda}\ne 0$. Choosing $v:= -\bar u_{\lambda}^- $ in \eqref{15} yields $ \Vert\nabla\bar u_{\lambda}^-\Vert_p^p+ \Vert\nabla\bar u_{\lambda}^- \Vert_q^q= 0$, namely $\bar u_{\lambda}^-= 0$. This forces $\bar u_\lambda\geq 0$ while, by \eqref{15} again, 
$$-\Delta_p\bar u_\lambda- \Delta_q\bar u_\lambda=\bar u_\lambda^{\tau-1}+\lambda c_0 \left(\bar u_\lambda^{\theta-1}+ \bar u_\lambda^{r-1}\right)\;\;\text{in}\;\;\Omega.$$
Lieberman's nonlinear regularity theory and Pucci-Serrin's maximum principle finally lead to $\bar u_{\lambda}\in{\rm int}(C_+)$. Now define, provided $(x,\xi)\in\Omega\times\R$,
\begin{equation*}\label{16}
\bar f_{\lambda}(x,\xi):=
\begin{cases}
(\xi^+)^{\tau-1}+ \lambda f(x,\xi^+) & \text{if }\xi\le\bar u_{\lambda}(x),\\
\bar u_{\lambda}(x)^{\tau-1}+ \lambda f(x,\bar u_{\lambda}(x)) & \text{otherwise},
\end{cases}
\quad\bar F_\lambda(x,\xi):=\int_0^\xi\bar f_\lambda(x,t)\, dt.
\end{equation*}
An easy verification ensures that the associated $C^1$-functional
$$\bar\Phi_{\lambda}(u):=\frac{1}{p}\Vert\nabla u \Vert_p^p+ \frac{1}{q} \Vert\nabla u\Vert_q^q- \into\bar F_{\lambda}(x, u(x))\, dx, \quad u\in X,$$
is coercive and weakly sequentially lower semicontinuous. So, it attains its infimum at some point $ u_{\lambda}\in X$. Assumption
$({\rm h}_1)$ produces
$$\bar\Phi_{\lambda}(u_{\lambda})<0=\bar\Phi_{\lambda}(0),$$
i.e., $u_\lambda\neq 0$, because $\tau<q < p$. As before, from
\begin{equation}\label{19} 
\langle A_p(u_{\lambda})+A_q(u_{\lambda}),v\rangle=\into\bar f_{\lambda}(x, u_{\lambda}(x)) v(x)\, dx \quad\forall\, v\in X
\end{equation}
we infer $u_{\lambda}\ge 0$. Test \eqref{19} with $v:=(u_{\lambda}-\bar u_{\lambda})^+ $, exploit $({\rm h}_1)$ again, and recall \eqref{15} to arrive at
\[
\begin{split}
\langle A_p(u_{\lambda})+A_q(u_{\lambda}), (u_{\lambda}-\bar u_{\lambda})^+ \rangle & = \into\left[\bar u_{\lambda}^{\tau-1}+ \lambda f(\cdot,\bar u_{\lambda})\right] (u_{\lambda}-\bar u_{\lambda})^+ dx \\
& \le\into\left[\bar u_{\lambda}^{\tau-1}+\lambda c_0(\bar u_{\lambda}^{\theta-1}+\bar u_{\lambda}^{r-1})\right]
(u_{\lambda}- \bar u_{\lambda})^+ dx\\
& = \langle A_p(\bar u_{\lambda})+A_q(\bar u_{\lambda}), (u_{\lambda}-\bar u_{\lambda})^+\rangle,
\end{split}
\]
which entails $u_{\lambda}\le\bar u_{\lambda}$ by monotonicity. Summing up, $u_{\lambda}\in[0,\bar u_{\lambda}]\setminus\{0\}$. On account of \eqref{19}, one thus has $u_{\lambda}\in S_{\lambda}$ for any $\lambda\in (0,\lambda_0)$. This completes the proof. 
\end{proof}
Our next result ensures that $ \mathcal L$ is an interval.
\begin{proposition}\label{prop6}
Let $({\rm h}_1)$ be satisfied. If $\hat\lambda\in\mathcal L$ then $(0,\hat\lambda)\subseteq\mathcal L$. 
\end{proposition}
\begin{proof}
Pick $\hat u\in S_{\hat\lambda}$, $\lambda\in(0,\hat\lambda)$, and define, provided $(x,\xi)\in\Omega\times\R$,
\begin{equation*}
\hat f_{\lambda}(x,\xi):=
\begin{cases}
(\xi^+)^{\tau-1}+\lambda f(x,\xi^+) & \text{if }\xi\le\hat u(x), \\
\hat u(x)^{\tau-1}+\lambda f(x,\hat u(x)) & \text{otherwise},
\end{cases}
\quad \hat F_\lambda(x,\xi):=\int_0^\xi\hat f_\lambda (x,t)\, dt.
\end{equation*}
The associated energy functional
$$\hat\Phi_{\lambda}(u):=\frac{1}{p}\Vert\nabla u\Vert_p^p+\frac{1}{q}\Vert\nabla u\Vert_q^q-\into\hat F_{\lambda}(x,u(x))\, dx, \quad u\in X,$$
turns out coercive, weakly sequentially lower semicontinuous, besides $C^1$. Now, arguing exactly as above yields the conclusion.
\end{proof}
A careful reading of this proof allows one to state the next `monotonicity' property. 
\begin{corollary}\label{cor7}
Under hypothesis $({\rm h}_1)$, for every $\hat\lambda\in\mathcal L$, $u_{\hat\lambda}\in S_{\hat\lambda}$, and  $\lambda\in (0,\hat\lambda)$ there exists $u_{\lambda}\in S_{\lambda}$ such that $u_{\lambda} \le u_{\hat\lambda}$. 
\end{corollary}
Actually, we can prove a more precise assertion.
\begin{proposition}\label{prop8}
Suppose $({\rm h}_1)$ and $({\rm h}_4)$ hold. Then to each $\hat\lambda\in\mathcal L$, $u_{\hat\lambda}\in  S_{\hat\lambda}$,
$\lambda\in(0,\hat\lambda)$ there corresponds $u_{\lambda}\in S_{\lambda}$ fulfilling
$u_{\hat\lambda}- u_{\lambda}\in{\rm int}(C_+)$.
\end{proposition}
\begin{proof}
Write $\rho:=\Vert u_{\hat\lambda}\Vert_{\infty}$. If $\mu_{\rho}$ is given by $({\rm h}_4)$ while $ u_{\lambda}$ comes from Corollary \ref{cor7} then
\begin{equation}\label{23}
\begin{split}
-\Delta_p u_{\hat\lambda} & - \Delta_q u_{\hat\lambda} +\lambda\mu_{\rho} u_{\hat\lambda}^{p-1}= u_{\hat\lambda}^{\tau-1}+ \hat\lambda f(x, u_{\hat\lambda})+ \lambda\mu_{\rho} u_{\hat\lambda}^{p-1}\\
&= u_{\hat\lambda}^{\tau-1}+\lambda f(x, u_{\hat\lambda})+\lambda\mu_{\rho} u_{\hat\lambda}^{p-1}+ (\hat\lambda-\lambda) f(x, u_{\hat\lambda})\\
& \ge u_{\lambda}^{\tau-1}+ \lambda f(x,u_{\lambda})+\lambda\mu_{\rho} u_{\lambda}^{p-1}=-\Delta_p u_{\lambda}- \Delta_q u_{\lambda}+\lambda\mu_{\rho} u_{\lambda}^{p-1}
\end{split}
\end{equation} 
because $u_\lambda\leq u_{\hat\lambda}$ and $f(x,t)\geq 0$ once $t\geq 0$. The function $h(x):=(\hat\lambda-\lambda) f(x, u_{\hat\lambda}(x))$ lies in $L^{\infty}(\Omega)$. Indeed, on account of $({\rm h}_1)$, we have
$$0\le h(x)\le c_0(\hat\lambda-\lambda)\left[\Vert u\Vert_{\infty}^{\theta-1}+\Vert u\Vert_{\infty}^{r-1}\right]\quad\forall\, x\in\Omega.$$
Pick any compact set $K\subseteq\Omega$. Recalling that $u_{\hat\lambda}\in{\rm int}(C_+)$ and using  $({\rm h}_1)$ again gives
$$h(x)\ge (\hat\lambda-\lambda)\left[ c_1 u_{\hat\lambda}(x)^{p-1}+c_2 u_{\hat\lambda}(x)^{q-1}\right]
\ge\left( c_1 \inf_{K} u_{\hat\lambda}^{p-1}+c_2 \inf_{K} u_{\hat\lambda}^{q-1}\right)>0,\;\; x\in\Omega,$$
whence $0\prec h$. Now, \eqref{23} combined with Proposition \ref{arcoya-ruiz} entail $u_{\hat\lambda}-u_{\lambda}
\in{\rm int}(C_+)$.
\end{proof}
The interval $\mathcal L$ turns out to be bounded.
\begin{proposition}\label{prop9}
Let $({\rm h}_1)$ and $({\rm h}_4)$ be satisfied. If $\lambda^*:=\sup\mathcal L$ then $\lambda^*< \infty$.
\end{proposition}
\begin{proof}
Fix $\lambda\in\mathcal L$, $u_{\lambda}\in S_{\lambda}$. Note that we can suppose $ \lambda> 1$, otherwise $ \mathcal L$ would be bounded, which of course entails $ \lambda^*< \infty $. Define
\begin{equation*}
g_{\lambda}(x,\xi):=
\begin{cases}
\lambda\left[ c_1(\xi^+)^{p-1}+c_2( \xi^+)^{q-1}\right] & \text{if }\xi\le u_{\lambda}(x),\\
\lambda\left[c_1 u_{\lambda}(x)^{p-1}+ c_2 u_{\lambda}(x)^{q-1}\right] & \text{otherwise}, 
\end{cases}
\; G_{\lambda}(x, \xi):=\int_0^\xi g_{\lambda}(x,t)dt
\end{equation*}
for every $(x,\xi)\in\Omega\times\R$, as well as
$$\Psi_\lambda(u):=\frac{1}{p}\Vert\nabla u \Vert_p^p+ \frac{1}{q}\Vert\nabla u \Vert_q^q- \into G_{\lambda}(x,u(x))\, dx,\quad
u\in X.$$
The same arguments employed before yield here a global minimum point, say $\bar u_\lambda$, to $\Psi_\lambda$. So, in particular,
\begin{equation}\label{barulambda}
\langle A_p(\bar u_\lambda)+A_q(\bar u_\lambda),v\rangle=\into g_{\lambda}(x, \bar u_\lambda(x))v(x)\, dx \quad\forall\, v \in X.
\end{equation}
Choosing $v:= -\bar u_\lambda^-$ first and then $v:= (\bar u_\lambda- u_{\lambda})^+ $ we obtain $\bar u_\lambda\in [0, u_{\lambda}]$; cf. the proof of Proposition \ref{prop5}. Since, by $({\rm p}_3)$ in Section \ref{prel}, $u_{\lambda},\phi_{1,q}\in{\rm int}(C_+)$, through \cite[Proposition 1]{MP} one has $t\phi_{1,q}\leq u_\lambda$, with $t>0$ small enough. Thus, on account of $({\rm p}_3)$ again,
\begin{equation*}
\begin{split}
\Psi_{\lambda}(t\phi_{1,q}) &= \frac{1}{p}\Vert\nabla(t\phi_{1,q})\Vert_p^p+\frac{1}{q} \Vert\nabla(t\phi_{1,q})\Vert_q^q
-\into G_{\lambda}(x,t\phi_{1,q}(x))\, dx\\
&=\frac{t^p}{p} \Vert\nabla\phi_{1,q}\Vert_p^p+\frac{t^q}{q}\Vert\nabla\phi_{1,q}\Vert_q^q-
\into\lambda\left(c_1\frac{t^p}{p}\phi_{1,q}^p+c_2 \frac{t^q}{q}\phi_{1,q}^q \right) dx\\
&= \frac{t^p}{p}\Vert\nabla\phi_{1,q}\Vert_p^p+\frac{t^q}{q}\lambda_{1,q}-\lambda c_1\frac{t^p}{p}\Vert\phi_{1,q}\Vert_p^p
-\lambda c_2\frac{t^q}{q}\\
& \le\frac{t^p}{p}\Vert\nabla\phi_{1,q}\Vert_p^p+ \frac{t^q}{q}\left(\lambda_{1,q}-\lambda c_2 \right) \\
& < \frac{t^p}{p}\Vert\nabla\phi_{1,q}\Vert_p^p+\frac{t^q}{q}\lambda_{1,q} (1-\lambda)= c_6 t^p- c_7 t^q.
\end{split}
\end{equation*}
Now, recall that $q<p$ and decrease $t$ when necessary to achieve 
$$\Psi_{\lambda}(\bar u_\lambda)=\min_X\Psi_\lambda\le\Psi_{\lambda}(t\phi_{1,q})< 0=\Psi_{\lambda}(0),$$
i.e., $\bar u_\lambda\neq 0$. Summing up, $\bar u_\lambda\in [0,u_\lambda]\setminus\{0\}$, whence, by \eqref{barulambda},  it turns out a positive solution of the equation
$$-\Delta_p u-\Delta_q u=\lambda c_1|u|^{p-2}u+\lambda c_2|u|^{q-2}u\quad\text{in}\quad\Omega.$$
Due to \cite[Theorem 2.4]{BT}, this prevents $\lambda$ from being arbitrary large, as desired.
\end{proof}
Le us finally prove that $\mathcal L=(0, \lambda^*]$. From now on, $\Phi_\lambda:X\to\R$ will denote the $C^1$-energy functional associated with problem \eqref{problem}. Evidently,
\begin{equation}\label{29'}
\Phi_{\lambda}(u)= \frac{1}{p} \Vert\nabla u\Vert_p^p+\frac{1}{q}\Vert \nabla u\Vert_q^q-\frac{1}{\tau}\Vert u^+ \Vert_{\tau}^{\tau}-\lambda \into F(x, u^+(x))\, dx \quad\forall\, u\in X.
\end{equation}
\begin{proposition}\label{prop10}
Under $({\rm h}_1)$, $({\rm h}_3)$, and $({\rm h}_4)$ one has $\lambda^* \in\mathcal L$.
\end{proposition}
\begin{proof}
Pick any $\{\lambda_n\}\subseteq (0, \lambda^*)$ fulfilling $\lambda_n\uparrow\lambda^*$. Via Corollary \ref{cor7}, construct a sequence $\{u_n\}\subseteq X$ such that $u_n\in S_{\lambda_n}$, $u_n\leq u_{n+1}$. Then
\begin{equation}\label{31}
\langle A_p(u_n)+A_q(u_n),v\rangle=\into u_n^{\tau-1}v\, dx+\lambda_n\into f(\cdot \,, u_n)v\, dx,\quad v\in X.
\end{equation}
We can also assume $\Phi_\lambda(u_n)<0$ (see the proof of Proposition \ref{prop5}), which means
\begin{equation}\label{30}
\Vert\nabla u_n\Vert_p^p+\frac{p}{q}\Vert \nabla u_n\Vert_q^q-\frac{p}{\tau}\Vert u_n\Vert_{\tau}^{\tau}-\lambda_n \into p F(x, u_n(x))\, dx< 0.
\end{equation}
Testing \eqref{31} with $v:= u_n$ gives
\begin{equation}\label{32}
\Vert\nabla u_n \Vert_p^p+\Vert\nabla u_n\Vert_q^q=\Vert u_n\Vert_{\tau}^{\tau}+\lambda_n\into f(\cdot \,,u_n) u_n\, dx.
\end{equation}
Since $q<p$ while $\lambda_1\leq\lambda_n$, from \eqref{30}--\eqref{32} it follows
\begin{equation}\label{33}
\into \left[f(\cdot \,,u_n) u_n- p F(\cdot \,,u_n)\right] dx\le\frac{1}{\lambda_1} \left(\frac{p}{\tau}-1\right)\Vert u_n \Vert_{\tau}^{\tau}
\quad\forall\, n\in\N.
\end{equation}
Observe next that, thanks to $({\rm h}_1)$ and $({\rm h}_3)$, one has
\begin{equation*}
f(x,\xi)\xi- p F(x,\xi)\ge c_8\xi^\beta-c_9\quad\text{in}\quad\Omega\times\R_+\, .
\end{equation*}
Consequently, \eqref{33} becomes 
$$c_8 \Vert u_n\Vert_{\beta}^{\beta}\le \frac{1}{\lambda_1}\left(\frac{p}{\tau}-1\right)\Vert u_n\Vert_\tau^\tau+c_{10}
\le c_{11}\Vert u_n\Vert_{\beta}^{\tau}+c_{10},\quad n\in\N,$$
because $\tau<\beta$. This clearly forces
\begin{equation}\label{35}
\Vert u_n\Vert_\beta\leq c_{12}\quad\forall\, n\in\N.
\end{equation}
If $r\leq\beta$ then $\{u_n\}$ turns out also bounded in $L^r(\Omega)$. Using \eqref{32} besides $({\rm h}_1)$ entails
\begin{equation}\label{38}
\begin{split}
\Vert u_n\Vert^p & \le\Vert\nabla u_n\Vert_p^p+\Vert\nabla u_n \Vert_q^q\le\Vert u_n \Vert_{\tau}^{\tau}
+\lambda^* \into f(\cdot \,, u_n) u_n\, dx \\
& \le\vert\Omega\vert^{1- \tau/r} \Vert u_n \Vert_r^{\tau}+ \lambda^* c_0 \into (u_n^{\theta}+ u_n^r)\, dx\\
& \le\vert\Omega\vert^{1- \tau/r} \Vert u_n \Vert_r^{\tau}+\lambda^* c_0 \into \left[(1+ u_n^r)+ u_n^r\right] dx,
\end{split}
\end{equation}
whence $\{u_n\}\subseteq X$ is bounded. Suppose now $\beta<r<p^*$. Two cases may occur.\\
\noindent 1) $p<N$. Let $t\in (0,1)$ satisfy
\begin{equation}\label{36}
\frac{1}{r}=\frac{1-t}{\beta}+\frac{t}{p^*}.
\end{equation}
The interpolation inequality \cite[p. 905]{GP1} yields $\Vert u_n\Vert_r\le\Vert u_n\Vert_{\beta}^{1-t}\Vert u_n \Vert_{p^*}^t$. Via \eqref{35} we thus obtain
\begin{equation}\label{37}
\Vert u_n \Vert_r^r \le c_{13}\Vert u_n\Vert_{p^*}^{tr},\quad n \in \N.
\end{equation}
Reasoning exactly as before and exploiting \eqref{37} produces
\begin{equation}\label{38}
\Vert u_n\Vert^p\leq\Vert\nabla u_n \Vert_p^p+\Vert\nabla u_n\Vert_q^q \le c_{14} \left(1+ \Vert u_n \Vert_{p^*}^{tr} \right)
\le c_{15}\left(1+\Vert u_n\Vert^{tr}\right).
\end{equation}
Finally, note that $tr<p$. Indeed, $ (r-p)\frac{N}{p}<\beta$ due to $({\rm h}_3)$, while
$$tr<p \;\iff\; \frac{r-\beta}{p^*-\beta}<\frac{p}{p^*}\;\iff\; (r-p)\frac{N}{p}<\beta;$$
cf. \eqref{36}. Now, the boundedness of $\{u_n\}\subseteq X$ directly stems from \eqref{38}.\\
2) $p\ge N$, which implies $p^*=+\infty$. We will repeat the previous argument with $p^*$ replaced by any $\sigma> r$. Accordingly, if $t\in (0,1)$ fulfills $\frac{1}{r}=\frac{1-t}{\beta}+ \frac{t}{\sigma}$ then $tr= \frac{\sigma(r-\beta)}{\sigma-\beta}$. Since, thanks to $({\rm h}_3)$ again,
$$\lim_{\sigma\to+\infty}\frac{\sigma (r-\beta)}{\sigma-\beta}= r-\beta< p,$$
one arrives at $tr< p$ for $\sigma$ large enough. This entails $\{u_n\}\subseteq X$ bounded once more.\\
Hence, in either case, we may assume
\begin{equation}\label{40}
u_n \rightharpoonup u^*\;\;\text{in}\;\; X\quad\text{and}\quad u_n \to u^* \;\;\text{in}\;\; L^r(\Omega),
\end{equation}
where a subsequence is considered when necessary. Testing \eqref{31} with $v:= u_n- u^*$ thus yields, as $n\to+\infty$,
$$\lim_{n\to+\infty}\langle A_p(u_n)+A_q(u_n), u_n- u^*\rangle= 0,$$
whence, by monotonicity of $A_q$, 
$$\limsup_{n\to+\infty}\left[\langle A_p(u_n), u_n- u^*\rangle+ \langle A_q(u), u_n- u^*\rangle \right]\le 0.$$
On account of \eqref{40} it follows
$$\limsup_{n\to+\infty}\langle A_p(u_n), u_n- u^*\rangle\le 0.$$
Recalling that $A_p$ enjoys the  $({\rm S})_+$-property, we infer $u_n\to u^* $ in $X$, besides $0\le u_n\le u^*$ for all $n\in\N$. Finally, let $n\to+\infty$ in \eqref{31} to get
$$\langle A_p(u^*)+A_q(u^*), v\rangle= \into (u^*)^{\tau-1}v\, dx+ \lambda^* \into f(\cdot \,, u^*)v\, dx\quad\forall\, v\in X,$$
i.e., $u^* \in S_{\lambda^*}$ and, a fortiori, $\lambda^*\in\mathcal L$. 
\end{proof}
Some meaningful (bifurcation) properties of the set $S_\lambda$ will now be established. 
\begin{proposition}\label{prop11}
Suppose $({\rm h}_1)$--$({\rm h}_4)$ hold true. Then, for every $\lambda\in(0,\lambda^*)$, problem \eqref{problem} admits two solutions $u_0,\hat u\in{\rm int}(C_+)$ such that $u_0\le\hat u$. Moreover, $ u_0 $ is a local minimizer of the associated energy functional $\Phi_{\lambda}$. 
\end{proposition}
\begin{proof}
Fix $\lambda\in(0,\lambda^*)$ and choose $\eta\in(\lambda,\lambda^*)$. By Proposition \ref{prop6}, there exists $u_\eta\in S_\eta$ while Proposition \ref{prop8} provides $u_0 \in S_{\lambda}$ satisfying
\begin{equation}\label{42}
u_0\in\text{int}_{C^1_0(\close)}([0, u_\eta]).
\end{equation}
The same reasoning adopted in the proof of Proposition \ref{prop6} ensures here that $u_0$ is a global minimum point to the functional
$$\Phi_{\lambda,\eta}(u):=\frac{1}{p}\Vert\nabla u\Vert_p^p+\frac{1}{q}\Vert\nabla u\Vert_q^q- \into F_{\lambda,\eta}(x,u(x))\, dx, \quad u\in X,$$
where $F_{\lambda,\eta}(x,\xi):=\int_0^\xi f_{\lambda,\eta} (x,t)\, dt$, with
\begin{equation*}
f_{\lambda,\eta}(x,\xi):=
\begin{cases}
(\xi^+)^{\tau-1}+\lambda f(x,\xi^+) & \text{if }\xi\le u_\eta(x), \\
u_\eta (x)^{\tau-1}+\lambda f(x,u_\eta (x)) & \text{otherwise.}
\end{cases}
\end{equation*}
By \eqref{42}, $u_0$ turns out a local $C^1_0(\close)$-minimizer of $\Phi_\lambda$, because $\Phi_\lambda\lfloor_{[0,u_\eta]}=\Phi_{\lambda,\eta}\lfloor_{[0,u_\eta]}$. Via Proposition \ref{prop2} we then see that this remains valid with $C^1_0(\close)$ replaced by $X$. Set
\begin{equation}\label{43}
f_0(x,\xi):=
\begin{cases}
u_0(x)^{\tau-1}+\lambda f(x, u_0(x)) & \text{if } \xi\le u_0(x), \\
\xi^{\tau-1}+\lambda f(x,\xi) & \text{otherwise,}
\end{cases}
\, \, F_0(x,\xi):=\int_0^\xi f_0(x,t)\, dt,
\end{equation}
$(x,\xi)\in\Omega\times\R$, as well as
\begin{equation}\label{43'}
\Phi_0(u):=\frac{1}{p}\Vert\nabla u\Vert_p^p+\frac{1}{q}\Vert\nabla u\Vert_q^q- \into F_0(x, u(x))\, dx \quad\forall\, u\in X.
\end{equation}
From \eqref{43} and the nonlinear regularity theory it follows $u_0\in K(\Phi_0)\subseteq [u_0)\cap{\rm int}(C_+)$. We may thus assume
\begin{equation}\label{45}
K(\Phi_0)\cap [u_0, u_{\eta}]= \{u_0\},
\end{equation}
or else a second solution of \eqref{problem} bigger than $u_0$ would exist. Bearing in mind the proof of Proposition \ref{prop10} and making small changes to accommodate the truncation at $u_0(x)$ shows that $\Phi_0$ satisfies condition (C). Let us next truncate
$f_0(x,\s)$ at $u_{\eta}(x)$ to construct a new Carath\'{e}odory function $\tilde f$, with primitive $\tilde F$ and associated functional $\tilde\Phi$, defined like in \eqref{43'} but replacing $F_0$ by $\tilde F$. Evidently,
$$K(\tilde\Phi)= K(\Phi_0)\cap [u_0,u_\eta],$$
whence $K(\tilde\Phi)=\{u_0\}$ because of \eqref{45}. Since $\tilde\Phi$ is coercive and weakly sequentially lower semicontinuous, it possesses a global minimum point that must coincide with $u_0$. An easy verification gives $\Phi_0\lfloor_{[0,u_\eta]}=\tilde\Phi\lfloor_{[0, u_\eta]}$. So, thanks to \eqref{42}, $u_0$ turns out a local $C_0^1(\close)$-minimizer of $\Phi_0$. This still holds when $X$ replaces $C^1_0(\close)$; cf. Proposition \ref{prop2}. We may suppose $K(\Phi_0)$ finite, otherwise infinitely many solutions of \eqref{problem} bigger than $u_0$ do exist. Adapting the argument exploited in \cite[Proposition 29]{APS} provides $\rho\in(0,1)$ such that 
\begin{equation}\label{50}
\Phi_0(u_0)< m_0:=\inf\{\Phi_0(u):\Vert u- u_0\Vert=\rho\}. 
\end{equation}
Finally, if $u\in{\rm int}(C_+)$ then simple calculations based on $({\rm h}_2)$ entail $\Phi_0(tu)\to-\infty$ as $t\to+\infty$. Therefore, the mountain pass theorem can be applied, and there is $\hat u\in X$ fulfilling
\begin{equation}\label{52}
\hat u\in K(\Phi_0),\quad\Phi_0(\hat u)\geq m_0.
\end{equation}
Via \eqref{50}--\eqref{52} one has $u_0\ne\hat u$ while the inclusion $K(\Phi_0)\subseteq [u_0)\cap{\rm int}(C_+)$  forces $u_0\leq\hat u$, which ends the proof.
\end{proof}
\begin{proposition}\label{prop12}
Under $({\rm h}_1)$--$({\rm h}_4)$, the solution set $S_\lambda$ admits  a smallest element $\bar u_\lambda$ for every $\lambda\in\mathcal L$. 
\end{proposition}
\begin{proof}
A standard procedure ensures that $ S_{\lambda} $ turns out downward directed; see, e.g., \cite[Section 4]{FP}. Lemma 3.10 at p. 178 of \cite{HP} yields 
\begin{equation}\label{hupap}
\essinf S_{\lambda}= \inf\{u_n:n\in\N\}
\end{equation}
for some decreasing sequence $\{u_n\}\subseteq S_{\lambda}$. Consequently, $0\le u_n\le u_1$ and
\begin{equation}\label{53}
\langle A_p(u_n)+A_q(u_n),v\rangle=\into\left[ u_n^{\tau-1}+\lambda f(\cdot \,,u_n)\right]v\, dx\quad\forall\, v\in X.
\end{equation}
Due to $({\rm h}_1)$, testing \eqref{53} with $v:=u_n$ we thus obtain 
\begin{equation*}
\begin{split}
\Vert u_n\Vert^p & \le\Vert\nabla u_n \Vert_p^p +\Vert \nabla u_n\Vert_q^q =\into\left[ u_n^{\tau}+\lambda f(\cdot \,, u_n) u_n \right] dx\\
& \le\into \left[u_n^{\tau}+ \lambda c_0 \left(u_n^{\theta}+ u_n^r\right) \right] dx 
\le\into \left[u_1^{\tau}+ \lambda c_0 \left(u_1^{\theta}+ u_1^r \right) \right] dx ,\quad n\in\N,
\end{split}
\end{equation*}
namely $\{u_n\}\subseteq X$ is bounded. Like before (cf. the proof of Proposition \ref{prop10}), this gives $u_n\to\bar u_\lambda$ in $X$, where a subsequence is considered if necessary. So, from \eqref{53} it easily follows
$$\langle A_p(\bar u_{\lambda})+A_q(\bar u_{\lambda}), v\rangle= \into\left[\bar u_{\lambda}^{\tau-1}+\lambda f(\cdot \,,\bar u_{\lambda}) \right] v\, dx\quad\forall\, v\in X.$$
Showing that $\bar u_{\lambda}\ne 0$ will entail $\bar u_{\lambda}\in S_{\lambda}$, whence the conclusion by \eqref{hupap}. To the aim, consider the problem
\begin{equation}\label{55}
-\Delta_p u-\Delta_q u= u^{\tau-1}\;\;\text{in}\;\;\Omega,\quad u> 0\;\;\text{in}\;\;\Omega,\quad u= 0\;\;\text{on}\;\;\rand.
\end{equation}
Its energy functional
$$\Phi_0(u):=\frac{1}{p}\Vert\nabla u\Vert_p^p+\frac{1}{q} \Vert\nabla u\Vert_q^q-\frac{1}{\tau}\Vert u^+\Vert_{\tau}^{\tau}, \quad u\in X,$$
turns out coercive and weakly sequentially lower semicontinuous. Hence, there exists $\tilde u\in X$ satisfying $\Phi_0(\tilde u)=\inf_X\Phi_0$. One has $u_0\neq 0$, because $\Phi_0(\tilde u)< 0=\Phi_0(0)$ (the argument is like in the proof of Proposition \ref{prop9}). Further, $\Phi_0'(\tilde u)=0$, i.e.,
$$\langle A_p(\tilde u)+A_q(\tilde u),v\rangle=\into(\tilde u^+)^{\tau-1}v\, dx\quad\forall\, v\in X.$$
Choosing $v:=-\tilde u^-$ we see that $u$ is a positive solution to \eqref{55}. Actually, $\tilde u\in{\rm int}(C_+)$ and, through a standard procedure \cite[Lemma 3.1]{GuMaPa}, $\tilde u$ turns out unique.\\
\textbf{Claim:} $\tilde u\le u$ for all $u\in S_{\lambda}$.\\
Indeed, fixed any $u\in S_{\lambda}$, define 
$$\Psi(w):=\frac{1}{p}\Vert\nabla u\Vert_p^p+\frac{1}{q}\Vert\nabla u\Vert_q^q- \into dx\int_0^{w(x)} g(x,t)\, dt ,\quad w\in X,$$
where
\begin{equation*}
g(x,t):=
\begin{cases}
(t^+)^{\tau-1} & \text{if } t\le u(x),\\
u(x)^{\tau-1} & \text{otherwise} 
\end{cases}
\quad\forall\, (x,t)\in\Omega\times\R.
\end{equation*}
The following assertions can be easily verified.
\begin{itemize}
\item $\Psi(u^*)=\inf_X\Psi$, with appropriate $u^*\in X$.
\item $\Psi(u^*)<0=\Psi(0)$, whence $u^*\ne 0$.
\item $u^*\in K(\Psi)\subseteq [0,u]\cap C_+$.
\end{itemize}
Therefore, $u^*$ is a positive solution of \eqref{55}. By uniqueness, this implies $u^*=\tilde u$. Thus, a fortiori, $\tilde u\le u$.\\
The claim brings $\tilde u\le u_n$, $n\in\N$, which in turn provides  $0<\tilde u\le \bar u_{\lambda}$, as desired.
\end{proof}
Let us finally come to some meaningful properties of the map
$$k:\lambda\in\mathcal L\mapsto \bar u_\lambda\in C^1_0(\close).$$ 
\begin{proposition}\label{prop13}
Suppose $({\rm h}_1)$--$({\rm h}_4)$ hold true. Then the function $k$ is both
\begin{itemize}
\item[$({\rm i}_1)$] strictly increasing, namely $\bar u_{\lambda_2}-\bar u_{\lambda_1}\in{\rm int}(C_+)$ if $\lambda_1<\lambda_2$, and
\item[$({\rm i}_2)$] left-continuous.
\end{itemize}
\end{proposition}
\begin{proof}
Pick $\lambda_1,\lambda_2\in \mathcal L$ such that $\lambda_1<\lambda_2$. Since $\bar u_{\lambda_2}\in S_{\lambda_2}$,  Proposition \ref{prop8} yields $u_{\lambda_1}\in S_{\lambda_1}$ fulfilling $\bar u_{\lambda_2}-u_{\lambda_1}\in{\rm int}(C_+)$, while Proposition \ref{prop12} entails $\bar u_{\lambda_1}\leq u_{\lambda_1}$. Hence, $\bar u_{\lambda_2}-\bar u_{\lambda_1}\in{\rm int}(C_+)$. This shows $({\rm i}_1)$.

If $\lambda_n\to\lambda^-$ in $\mathcal L$ then, by $({\rm i}_1)$, the  sequence $\{\bar u_{\lambda_n}\}$ turns out increasing. Its boundedness in $X$ immediately stems from $({\rm h}_1)$; see the previous proof. Now, repeat the argument below \eqref{40} to arrive at
\begin{equation}\label{59}
\bar u_{\lambda_n}\to\tilde u_{\lambda}\;\;\text{in}\;\; X,
\end{equation}
whence $\tilde u_{\lambda}\in S_{\lambda}\subseteq{\rm int}(C_+)$. We finally claim that $\tilde u_{\lambda}=\bar u_{\lambda}$. Assume on the contrary
\begin{equation}\label{60}
\bar u_{\lambda}(x_0)<\tilde u_{\lambda}(x_0)\;\;\text{for some}\;\; x_0\in\Omega.
\end{equation}
Lieberman's nonlinear regularity theory gives $\{\bar u_n\}\subseteq C^{1,\alpha}_0(\close)$ as well as 
$$\Vert\bar u_{\lambda_n} \Vert_{C^{1, \alpha}_0(\close)}\le c_{16}\quad\forall\, n\in\N.$$
Since the embedding $C^{1,\alpha}_0(\close)\hookrightarrow C^1_0(\close)$ is compact, \eqref{59} becomes
$$\bar u_{\lambda_n}\to\tilde u_{\lambda}\;\;\text{in}\;\; C^1_0(\close).$$
Because of \eqref{60}, this implies $\bar u_{\lambda}(x_0)<\bar u_{\lambda_n}(x_0)$ for any $n$ large enough, against $({\rm i}_1)$. Consequently, $\tilde u_{\lambda}=\bar u_{\lambda}$, and $({\rm i}_2)$ follows from \eqref{59}.
\end{proof}
Gathering Propositions \ref{prop5}--\ref{prop13} together we obtain the following
\begin{theorem}
Let $({\rm h}_1)$--$({\rm h}_4)$ be satisfied. Then, there exists $\lambda^*>0 $ such that problem \eqref{problem} admits
\begin{itemize}
\item[$({\rm j}_1)$] at least two solutions $u_0,\hat u\in {\rm int}(C_+)$, with $u_0\le\hat u$, for every $\lambda\in(0,\lambda^*)$,
\item[$({\rm j}_2)$] at least one solution $u^* \in{\rm int}(C_+)$ when $\lambda=\lambda^*$,
\item[$({\rm j}_3)$] no positive solutions for all $\lambda>\lambda^*$,
\item[$({\rm j}_4)$] a smallest positive solution $\bar u_{\lambda}\in{\rm int}(C_+)$ provided $\lambda\in (0, \lambda^*]$.
\end{itemize}
Moreover, the map $\lambda\in(0,\lambda^*]\mapsto\bar u_{\lambda}\in C_0^1(\close)$ is strictly increasing and left-continuous.
\end{theorem}
\vskip3pt
\noindent{\sc Acknowledgment.} This work is performed within the 2016--2018 Research Plan - Intervention Line 2: `Variational Methods and Differential Equations', and partially supported by GNAMPA of INDAM.

\begin{thebibliography}{40}

\bibitem{APS} S. Aizicovici, N.S. Papageorgiou and V. Staicu, \emph{Degree theory for operators of monotone type and nonlinear elliptic equations with inequality constraints}, Mem. Amer. Math. Soc. {\bf 196} (2008).

\bibitem{ABC} A. Ambrosetti, H. Brezis, and G. Cerami, \emph{Combined effects of concave and convex nonlinearities in some elliptic problems}, J. Funct. Anal. {\bf 122} (1994), 519--543 

\bibitem{AR}{D. Arcoya and D. Ruiz, \emph{The Ambrosetti-Prodi problem for the $p$-Laplace operator}, Comm. Partial Differential Equations {\bf 31} (2006), 849--865.}

\bibitem{BLP}{G. Barletta, R. Livrea, and N.S. Papageorgiou, \emph{Bifurcation phenomena for the positive solutions of semilinear elliptic problems with mixed boundary conditions}, J. Nonlinear Convex Anal. {\bf 17} (2016), 1497--1516.}

\bibitem{BT}{V. Bobkov and M. Tanaka, \emph{On positive solutions for $(p,q)$-Laplace equations with two parameters}, Calc. Var. Partial Differential Equations {\bf 54} (2015),  3277--3301.}

\bibitem{BoTa} V. Bobkov and M. Tanaka, \emph{Remarks on minimizers for $(p, q)$-Laplace equations with two parameters}, Commun. Pure Appl. Anal. {\bf 17} (2018), 1219--1253.

\bibitem{CI}{L. Cherfils and Y. Ilyasov, \emph{On the stationary solutions of generalized reaction diffusion equations with $(p,q)$-Laplacian}, Comm. Pure Appl. Anal. {\bf 4} (2005), 9--22.}

\bibitem{De}{G. H. Derrick, \emph{Comments on nonlinear wave equations as models for elementary particles}, J. Math. Phys. \textbf{5} (1964), 1252--1254.}

\bibitem{Fi}{P.C. Fife, \emph{Mathematical Aspects of Reacting and Diffusing Systems}, Lect. Notes in Biomath. \textbf{28}, Springer, Berlin, 1979.}

\bibitem{FP}{M. Filippakis and N.S. Papageorgiou, \emph{Multiple constant sign and nodal solutions for nonlinear elliptic equations with the $p$-Laplacian}, J. Differential Equations {\bf 245} (2008),  1883--1922.}

\bibitem{GMP}{J. Garcia Azorero, J. Manfredi, and I. Peral Alonso, \emph{Sobolev versus H$\ddot{\text{o}}$lder local minimizers and global multiplicity for some quasilinear elliptic equations}, Comm. Contemp. Math. {\bf 2} (2000), 385--404.}

\bibitem{GP1}{L. Gasi\'nski and N.S. Papageorgiou, \emph{Nonlinear Analysis}, Chapman \& Hall, CRC, Boca Raton, Fl, 2006.}

\bibitem{GP-SVAN} L. Gasi\'nski and N.S. Papageorgiou, \emph{Multiple solutions for nonlinear coercive problems with a nonhomogeneous differential operator and a nonsmooth potential}, Set-Valued Anal. {\bf 20} (2012), 417--443.

\bibitem{GP}{L. Gasi\'nski and N.S. Papageorgiou, \emph{Bifurcation-type results for nonlinear parametric elliptic equations}, Proc. Royal Soc. Edinburgh Sect. A {\bf 142} (2012), 595--623.}

\bibitem{GuMaPa} {U. Guarnotta, S.A. Marano, and N.S. Papageorgiou, \emph{Multiple nodal solutions to a Robin problem with sign-changing potential and locally defined reaction}, Atti Accad. Naz. Lincei Rend. Lincei Mat. Appl., in press. }

\bibitem{GZ}{Z. Guo and Z. Zhang, \emph{$W^{1,p}$ versus $C^1$ local minimizers and multiplicity results for quasilinear elliptic equations}, J. Math. Anal. Appl. {\bf 286} (2003),  32--50.}

\bibitem{HP}{S. Hu and N.S. Papageorgiou, \emph{Handbook of Multivalued Analysis. Volume I: Theory}, Kluwer Academic Publishers, Dordrecht, The Netherlands, 1997.}

\bibitem{L}{G. Lieberman, \emph{The natural generalization of the natural conditions of Ladyshenskaya and Ural'tseva for elliptic equations}, Comm. Partial Differential Equations {\bf 16} (1991), 311--361.}

\bibitem{MM}{S.A. Marano and S. Mosconi, \emph{Some recent results on the Dirichlet problem for $(p,q)$-Laplace equations}, Discrete Contin. Dyn. Syst. Ser. S {\bf 11} (2018),  279--291.}

\bibitem{MMP}{S.A. Marano, S. Mosconi, and N.S. Papageorgiou, \emph{Multiple solutions to $(p,q)$-Laplacian problems with resonant concave nonlinearity}, Adv. Nonlinear Stud. {\bf 16} (2016), 51--65.}

\bibitem{MMP1}{S.A. Marano, S. Mosconi, and N.S. Papageorgiou, \emph{On a $(p,q)$-Laplacian problem with concave and asymmetric perturbation}, Atti Accad. Naz. Lincei Rend. Lincei Mat. Appl. {\bf 29} (2018), 109--125.}

\bibitem{MP}{S.A. Marano and N.S. Papageorgiou, \emph{Positive solutions to a Dirichlet problem with $p$-Laplacian and concave-convex nonlinearity depending on a parameter}, Comm. Pure Appl. Anal. {\bf 12} (2013),  815--829.}

\bibitem{PR}{N.S. Papageorgiou and V.D. Radulescu, \emph{Bifurcation of positive solutions for nonlinear nonhomogeneous Robin and Neumann problems with competing nonlinearities}, Discrete Contin. Dyn. Syst.-A {\bf 35} (2016), 5008--5036.}

\bibitem{PRR}{N.S. Papageorgiou, V.D. Radulescu, and D. Repovs, \emph{Robin problems with indefinite linear part and competition phenomena}, Commun. Pure Appl. Anal. {\bf 16} (2017),  1293--1314.}

\bibitem{PS}{P. Pucci and J. Serrin, \emph{The Maximum Principle}, Birkh$\ddot{\text{a}}$user, Basel, 2007.}

\bibitem{W}{H. Wilhelmson, \emph{Explosive instabilities of reaction-diffusion equations}, Phys. Rev. A {\bf 36} (1987),  965--966.}

\bibitem{Z}{V.V. Zhikov, \emph{Averaging of functionals of the calculus of variations and elasticity theory}, Izv. Akad. Nauk SSSR Ser. Mat. {\bf 50} (1986), 675--710 ; English translation in Math. USSR-Izv. {\bf 29} (1987),  33--66.}

\end{thebibliography}
\end{document}